\documentclass[preprint]{elsarticle}
\usepackage[utf8]{inputenc}

\usepackage{lineno,hyperref}
\modulolinenumbers[1]

\usepackage{amssymb,amsmath,amsthm}
\usepackage{mathtools}

\usepackage{paralist}

\usepackage[all,color]{xy}
\usepackage{pdfpages}
\usepackage{setspace}

\usepackage{color,framed}
\usepackage{graphicx}
\usepackage{caption}

\large

\theoremstyle{plain}
\newtheorem{theorem}{Theorem}
\newtheorem{proposition}[theorem]{Proposition}
\newtheorem{lemma}[theorem]{Lemma}

\theoremstyle{definition}
\newtheorem{definition}[theorem]{Definition}

\newtheorem{remark}[theorem]{Remark}
\newtheorem{notation}[theorem]{Notation}
\theoremstyle{remark}
\newtheorem{claim}[theorem]{Claim}

\DeclareMathOperator{\lcm}{lcm}

\begin{document}

\begin{frontmatter}
\title{Vertex-Pancyclism in the Generalized Sum of Digraphs}
\tnotetext[t1]{This research was supported by grants CONACYT FORDECYT-PRONACES CF-2019/39570 and UNAM-DGAPA-PAPIIT IN102320.}

\author[imunam]{Narda Cordero-Michel\corref{cor1}}
\ead{narda@matem.unam.mx}
\author[imunam]{Hortensia Galeana-S\'{a}nchez}
\ead{hgaleana@matem.unam.mx}

\cortext[cor1]{Corresponding author}
\address[imunam]{Instituto de Matem\'aticas, Universidad Nacional Aut\'onoma de M\'exico, Ciudad Universitaria, CDMX, 04510, M\'exico}


\begin{abstract}
A digraph $D=(V(D)$, $A(D))$ of order $n\geq 3$ is \emph{pancyclic}, whenever $D$ contains a directed cycle of length $k$ for each $k\in \{3,\ldots,n\}$; and $D$ is \emph{vertex-pancyclic} iff, for each vertex $v\in V(D)$ and each  $k\in \{3,\ldots,n\}$, $D$ contains a directed cycle of length $k$ passing through $v$. 

Let $D_1$, $D_2$,  \ldots, $D_k$ be a collection of pairwise vertex disjoint digraphs. The \emph{generalized sum} (g.s.) of $D_1$, $D_2$, \ldots, $D_k$, denoted by $\oplus_{i=1}^k D_i$ or $D_1\oplus D_2 \oplus \cdots \oplus D_k$, is the set of all digraphs $D$ satisfying:
(i) $V(D)=\bigcup_{i=1}^k V(D_i)$,
(ii) $D\langle V(D_i) \rangle \cong D_i$ for $i=1,2,\ldots, k$, and
(iii) for each pair of vertices belonging to different summands of $D$, there is exactly one arc between them,
with an arbitrary but fixed direction.
A digraph $D$ in $\oplus_{i=1}^k D_i$ will be called a \emph{generalized sum} (g.s.) of $D_1$, $D_2$,  \ldots, $D_k$.

Let $D_1$, $D_2$, \ldots, $D_k$ be a collection of $k$ pairwise vertex disjoint Hamiltonian digraphs, in this paper we give simple sufficient conditions for a digraph $D\in \oplus_{i=1}^k D_i$ be vertex-pancyclic. 
This result extends a result obtained by Cordero-Michel, Galeana-Sánchez and  Goldfeder in 2016.
\end{abstract}

\begin{keyword}
digraph, generalizations of tournaments, pancyclic digraph
\end{keyword}

\end{frontmatter}

\section{Introduction}
\label{sec:introduction}
Given a digraph $D=(V(D),A(D))$, the problem of determining  if $D$ becomes a pancyclic or a vertex-pancyclic digraph, has been studied by several authors, see by example \cite{Bang-Jensen1999313, Bang-Jensen2009, Bang-Jensen1995141, Gutin1995153, Moon1966297, Randerath2002219, Thomassen197785}.

A digraph $D$ is said to be a \emph{tournament} (respectively, a \emph{semicomplete digraph}) whenever for each pair of different vertices, there is exactly one arc (resp. at least one arc) between them. A \emph{$k$-hypertournament} $H$ on $n$ vertices, where $2\leq k\leq n$, is a pair $H=(V_H, A_H)$, where $V_H$ is the vertex set of $H$ and $A_H$ is a set of $k$-tuples of vertices such that, for all subsets $S\subseteq V_H$ with $\vert S\vert=k$, $A_H$ contains exactly one permutation of $S$.
In \cite{Moon1966297}, Moon proved that every strong tournament is vertex-pancyclic; in \cite{Bang-Jensen2009}, Bang-Jensen and Gutin state that this result can be extended for strong semicomplete digraphs; and in \cite{Li20132749}, Li \emph{et al.} extended Moon's result to $k$-hypertournaments on $n$ vertices, where $3\leq k\leq n-2$. 

Other results on pancyclism involve large degrees of the vertices. 
For example in \cite{Randerath2002219}, Randerath \textit{et al.} proved that every digraph $D$ on $n\geq 3$ vertices for which $\min\{\delta^+(D),\delta^-(D)\} \geq \frac{n+1}{2}$ is vertex-pancyclic. 
In \cite{Thomassen197785}, Thomassen proved that if $D$ is a strong digraph on $n$ vertices, such that $d(x) + d(y) \geq 2n$ is satisfied for each pair of non-adjacent vertices $x$ and $y$, then either $D$ has directed cycles of all lengths 2, 3, \ldots, $n$, or $D$ is a tournament (in which case it has cycles of all lengths 3, 4, \ldots, $n$), or $n$ is even and D is isomorphic to a complete bipartite digraph whose partition sets have $n/2$ vertices.
Continuing in this direction, in \cite{Bang-Jensen1999313},  Bang-Jensen and Guo proved that any digraph $D$ with no symmetric arcs, $n\geq 9$, minimum degree $n-2$ and such that for each pair of non-adjacent vertices $x$ and $y$ the inequality $d^+_D(x)+d^-_D(y)\geq n-3$ holds, is vertex-pancyclic. 


A digraph $D$ is called a \emph{complete $k$-partite} or \emph{multipartite digraph} (CMD) if its vertex set can be partitioned into $k$ stable  sets (\emph{partite sets}) and for any two vertices $x$ and $y$ in different partite sets either
$(x, y)$ or $(y, x)$ (or both) is an arc (are arcs) of $D$. A CMD digraph $D$ is called \emph{ordinary} if for any pair $X$, $Y$ of its partite sets, the set of arcs with both end vertices in $X\cup Y$ coincides with $X\times Y=\{(x,y): x\in X, y\in Y\}$ or $Y\times X$ or $(X \times Y) \cup (Y \times X)$. 

An ordinary CMD $D$ is called a \emph{zigzag digraph} if it has more than four vertices and $k\geq 3$ partite sets $V_1$, $V_2$, $V_3$, \ldots, $V_k$ such that $A(V_2, V_1)=A(V_i, V_2)=A(V_1 , V_i)=\emptyset$ for any $i\in\{3, 4,\ldots, k\}$, $|V_1|=|V_2|=|V_3|+|V_4|+\cdots+|V_k|$. 

In \cite{Gutin1995153}, Gutin characterized pancyclic and vertex-pancyclic ordinary complete $k$-partite digraphs:
\begin{theorem}[Gutin \cite{Gutin1995153}]
\begin{enumerate}[(1)] 
\item An ordinary complete $k$-partite digraph ($k\geq 3$) $D$ is pancyclic if and only if:
\begin{enumerate}[(i)] 
\item $D$ is strongly connected;
\item it has a spanning subdigraph consisting of a family of vertex disjoint cycles;
\item it is neither a zigzag digraph nor a 4-partite tournament with at least five vertices. 
\end{enumerate}
\item A pancyclic ordinary complete $k$-partite digraph D is vertex-pancyclic if and only if either:
\begin{enumerate}[(i)] 
\item $k > 3$ or
\item $k=3$ and $D$ has two 2-cycles $Z_1$, $Z_2$ such that $V(Z_1)\cup V(Z_2)$ contains vertices in exactly three partite sets.
\end{enumerate}
\end{enumerate}
\end{theorem}
Notice that a g.s. of digraphs with no arcs, where all exterior arcs between two summands have the same direction, is an ordinary CMD. 
In our results we also work with a partition of the vertex set of a digraph, but instead of asking that each partite set be independent we ask for each partite set to have a Hamiltonian cycle; we allow that the arcs between two partite sets have any direction and we forbid symmetric arcs. In this way, our problem has similarities with Gutin's problem but they are different problems.   



In \cite{Bang-Jensen1995141}, Bang-Jensen and Huang characterized pancyclic and vertex-pancyclic quasi-transitive digraphs. 
In \cite{Yeo1999137}, Yeo proved that every regular multipartite tournament with at least 5 partite sets is vertex-pancyclic and, in \cite{Yeo2007949}, he proved that all regular 4-partite tournaments with at least 13918 vertices are vertex-pancyclic.

In \cite{Cordero-Michel20161763}, we proved that a strong generalized sum of $k$ vertex disjoint Hamiltonian graphs is Hamiltonian. Moreover, given the relationship between digraphs and 2-edge-colored graphs, in \cite{Cordero-Michel2020} we adapted the definition of generalized sum of digraphs to 2-edge-colored graphs and we obtained sufficient conditions for a colored generalized sum of $k$ vertex disjoint 2-edge-colored graphs, each containing an alternating Hamiltonian cycle, to be a vertex alternating-pancyclic graph, this is, to contain an alternating cycle of each even length between 4 and the number of vertices of the graph.
In this paper we give simple sufficient conditions for a  generalized sum of $k$ vertex disjoint Hamiltonian digraphs to be a vertex-pancyclic digraph, in this way we extend the result of \cite{Cordero-Michel20161763} and we obtain a result on vertex-pancyclism for digraphs, where cycles of odd and even lengths are considered. 

\section{Definitions}
\label{sec:definitions}
In this paper $D=(V(D),A(D))$ will denote a \emph{digraph}. An arc $(u,v)\in A(D)$ will also be denoted by $u\to v$. Two different vertices $u$ and $v$ are adjacent if $u\to v$ or $v\to u$. 
Let $A$ and $B$ be two sets of vertices or subdigraphs of a digraph $D$, we define the set of arcs $(A,B)$, as the set of all arcs with tail in $A$ (or in the vertex set of $A$) and head in $B$ (or in the vertex set of $B$). If $A=\{a\}$ or $B=\{b\}$, we use the notation $(a,B)$ or $(A,b)$,  respectively, instead of $(A,B)$. 
Also, we denote by $A\to B$ whenever for each vertex $a$ in $A$ and each vertex $b$ in $B$ we have $a\to b$, and we denote by $A\mapsto B$ whenever $A\to B$ and $(B,A)$ is empty.  If $A=\{a\}$ or $B=\{b\}$, we use the notation $a\to B$ or $A\to b$,  respectively, instead of $A\to B$ and $a\mapsto B$ or $A\mapsto b$,  respectively, instead of $A\mapsto B$. 

The subdigraph induced by a set of vertices $U\subseteq V(D)$ will be denoted by 
$ D\langle U\rangle$;
and if $H$ is a subdigraph of $D$, the subdigraph induced by $V(H)$ will be denoted simply by 
$D\langle H\rangle$. 

A digraph is \emph{strong} whenever for each pair of different vertices $u$ and $v$, there exist a $uv$-path and a $vu$-path.

A \emph{spanning subdigraph} $E$ of $D$ is a subdigraph of $D$ such that $V(E)=V(D)$. We say that $E$ spans $D$.

Our paths and cycles are always directed. A \emph{cycle-factor} of a digraph $D$ is a collection $\mathcal{F}$ of pairwise vertex disjoint cycles in $D$ such that each vertex of $D$ belongs to a member of $\mathcal{F}$. A cycle-factor consisting of $k$ cycles is a \emph{$k$-cycle-factor}. We denote $\mathcal{F}$ by $\mathcal{F}=C_1 \cup \cdots \cup C_k$. 

For further details we refer the reader to \cite{Bang-Jensen2009}.

\begin{definition}
\label{definition g.s.}
Let $D_1$, $D_2$,  \ldots, $D_k$ be a collection of pairwise vertex disjoint digraphs. The \emph{generalized sum} (g.s.) of $D_1$, $D_2$, \ldots, $D_k$, denoted by $\oplus_{i=1}^k D_i$ or $D_1\oplus D_2 \oplus \cdots \oplus D_k$, is the collection of all digraphs $D$ satisfying:
(i) $V(D)=\bigcup_{i=1}^k V(D_i)$,
(ii) $D\langle V(D_i) \rangle \cong D_i$ for $i=1,2,\ldots, k$; and
(iii) between each pair of vertices in different summands of $D$ there is exactly one arc, with an arbitrary but fixed direction.

Let $D_1$, $D_2$, \ldots, $D_k$ be a collection of vertex disjoint digraphs. A digraph $D$ in $\oplus_{i=1}^k D_i$ will be called a \emph{generalized sum} (g.s.) and we will say that $e\in A(D)$ is an \emph{interior arc} of $D$ iff $e\in \bigcup_{i=1}^k A(D_i)$ and an \emph{exterior arc} of $D$ otherwise. The set of the exterior arcs of $D$ will be denoted by $E_\oplus$.
\end{definition}

\begin{remark}
Clearly the g.s. of two vertex disjoint digraphs if well defined and is commutative.
Let $D_1$, $D_2$, $D_3$ be three vertex disjoint digraphs. It is easy to see that the sets $(D_1\oplus D_2)\oplus D_3$ defined as $\bigcup_{D\in D_1\oplus D_2} D\oplus D_3 $ and $D_1 \oplus (D_2\oplus D_3)$ defined as $\bigcup_{D'\in D_2\oplus D_3} D_1\oplus D'$ are equal, thus $\oplus_{i=1}^3 D_i=(D_1\oplus D_2)\oplus D_3 = D_1\oplus(D_2\oplus D_3)$ is well defined. By means of an inductive process it is easy to see that the g.s. of $k$ vertex disjoint digraphs is well defined, and is associative and  commutative.
\end{remark}

\begin{notation}
Let $k_1$ and $k_2$ be two positive integers, where $k_1\leq k_2$. We will denote by $[k_1,\ k_2]$ the set of integers $\{k_1$, $k_1+1$, \ldots, $k_2\}$.
\end{notation}

\begin{remark}[\cite{Cordero-Michel20161763}]
\label{remark subsuma}
Let $D_1$, $D_2$, \ldots, $D_k$ be a collection of pairwise vertex disjoint digraphs, $D\in\oplus_{i=1}^k D_i$ and $J\subset [1, k]$. The induced subdigraph of $D$ by $\bigcup_{j\in J} V(D_j)$, $H=D\langle \bigcup_{j\in J} V(D_j) \rangle$, belongs to the g.s. of $\{D_j\}_{j\in J}$.
\end{remark}

\begin{figure}
\begin{center}   
\includegraphics[scale=0.9]{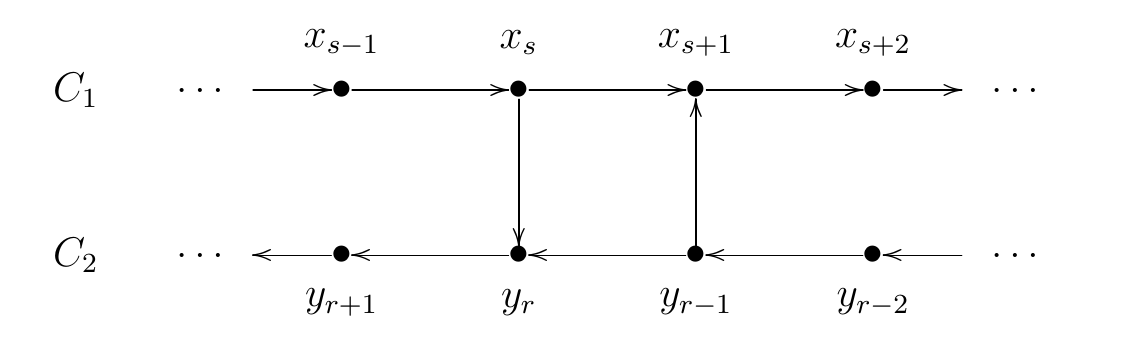}
\caption{A good pair of arcs.}
\label{fig:good pair}
\end{center} 
\end{figure}

\begin{definition}
Let $D$ be a digraph and let $C_1=(x_0$, $x_1$, \ldots, $x_{n-1}$, $x_0)$ and $C_2=(y_0$, $y_1$, \ldots, $y_{m-1}$, $y_0)$ be two vertex disjoint cycles in $D$. A pair of arcs $x_s\to y_r$, $y_{r-1}\to x_{s+1}$ with $s\in [0,\ n-1]$, $r\in [0,\ m-1]$ (subscripts are taken modulo $n$ and $m$, respectively),  is  a \emph{good pair of arcs} (Figure \ref{fig:good pair}).

Whenever there is a good pair of arcs between two vertex disjoint cycles $C_1$ and $C_2$, we simply say that there is a good pair.
\end{definition}

\section{Preliminary results}

In this section we prove properties of strong digraphs in the g.s. of two Hamiltonian digraphs which allow us to find cycles of several lengths.

\begin{proposition}[Galeana-S\'{a}nchez and Goldfeder \cite{Galeana2014315}]
\label{propo merging cycles with a good pair of arcs}
Let $C_1$ and $C_2$ be two disjoint cycles in a digraph $D$. If there is a good pair between them, then there is a cycle with vertex set $V(C_1)\cup V(C_2)$ \textnormal{(Figure \ref{fig:good pair of arcs})}.
\end{proposition}
%

\begin{figure}
\begin{center}   
\includegraphics[scale=0.9]{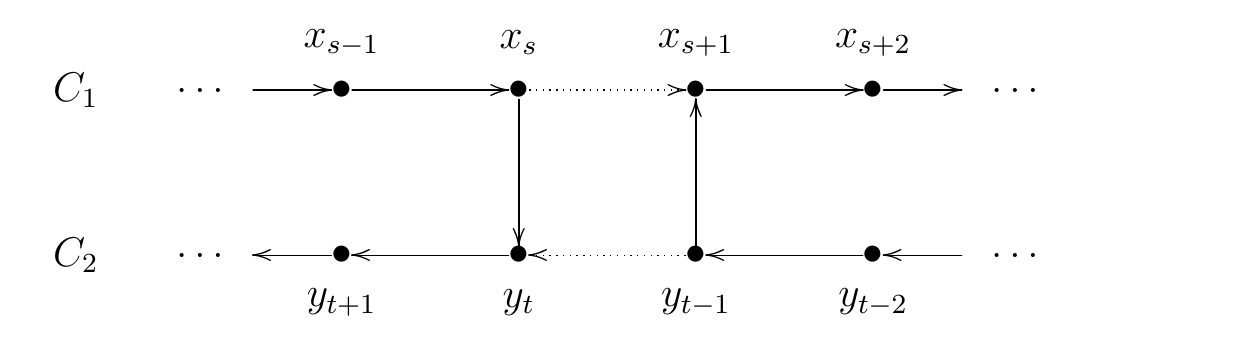}
\caption{A cycle using a good pair of arcs.}
\label{fig:good pair of arcs}
\end{center} 
\end{figure}
 
\vspace{6mm}
In what follows, the subscripts will be taken modulo $n$ (resp. modulo $m$) for vertices in $C_1 = (x_0$, $x_1$, \ldots,  $x_{n-1}$, $x_0)$ (resp. in $C_2 = (y_0$, $y_1$, \ldots, $y_{m-1}$, $y_0)$).
\vspace{6mm}

\begin{definition}
\label{definition parallel}
\begin{sloppypar}
Let $D_1$ and $D_2$ be two vertex disjoint digraphs with Hamiltonian cycles $C_1=(x_0$, $x_1$, \ldots, $x_{n-1}$, $x_0)$ and $C_2=(y_0, y_1, \ldots, y_{m-1}, y_0)$, respectively, and let $D\in D_1\oplus D_2$ be a generalized sum. For an exterior arc $x_s\to y_t$ (respectively $y_t\to x_s$), we define the \emph{parallel class} of $x_s\to y_t$ (resp. $y_t\to x_s$), denoted by $P_{x_s\to y_t}$ (resp. $P_{y_t\to x_s}$), as follows:  $P_{x_s\to y_t}=\{x_{s+i}\to y_{t-i}\ |\ 0\leq i\leq l-1\}$ (resp. $P_{y_t\to x_s}=\{y_{t+i}\to x_{s-i}\ |\ 0\leq i\leq l-1\}$), where $l=\lcm(n,\ m)$; and two exterior arcs $a_1, a_2$ will be named \emph{parallel} whenever there exists an exterior arc $z\to w$ such that $\{a_1, a_2\} \subset P_{z\to w}$.
\end{sloppypar}
\end{definition}

\begin{remark}
\label{remark: existence parallel class}
Let $D_1$ and $D_2$ be two vertex disjoint digraphs with Hamiltonian cycles $C_1=(x_0$, $x_1$, \ldots, $x_{n-1}$, $x_0)$ and $C_2=(y_0, y_1, \ldots, y_{m-1}, y_0)$, respectively, and let $D\in D_1\oplus D_2$ be a generalized sum. If $D$ has no good pair, then for each exterior arc $x_s\to y_t$ (resp. $y_t\to x_s$),\footnote{Recall that, by definition, there is exactly one arc between $x_s$ and $y_t$. So, $x_s\to y_t$ and $y_t\to x_s$ cannot exist simultaneously.} $x_s\in V(C_1)$ and $y_t\in V(C_2)$, $P_{x_s\to y_t}$ (resp. $P_{y_t\to x_s}$) exists and is contained in $D$. 
\end{remark} 
\begin{proof}
We construct a sequence of exterior arcs as follows: (i) Since $x_s\to y_t$ (resp. $y_t\to x_s$) is an arc in $D$ and $D$ has no good pair, we have that $x_{s+1}\to y_{t-1}$ (resp. $y_{t+1}\to x_{s-1}$). (ii) Assume that $x_s\to y_t$, $x_{s+1}\to y_{t-1}$, \ldots, $x_{s+i}\to y_{t-i}$ (resp. $y_t\to x_s$, $y_{t+1}\to x_{s-1}$, \ldots, $y_{t+i}\to x_{s-i}$) have been constructed. (iii) Now, as $D$ has no good pair, we have $x_{s+(i+1)}\to y_{t-(i+1)}$ (resp. $y_{t+(i+1)}\to x_{s-(i+1)}$). (iv) The sequence finishes the first time that $x_{s+k}\to y_{t-k}=x_s\to y_t$ (resp. $y_{t+k}\to x_{s-k}=y_t\to x_s$). That is when $k=\lcm(n,\ m)$. 
\end{proof}

Notice that, since $l=\lcm(n,\ m)$ we have that for each vertex $w\in V(C_1)\cup V(C_2)$ and each $(u, v)\in E_\oplus$, there exists at least one arc in the parallel class $P_{u\to v}$  incident with $w$.

\begin{remark}
\label{remark parallel class}
Let $D_1$ and $D_2$ be two vertex disjoint digraphs with Hamiltonian cycles $C_1=(x_0$, $x_1$, \ldots, $x_{n-1}$, $x_0)$ and $C_2=(y_0, y_1, \ldots, y_{m-1}, y_0)$, respectively, and let $D\in D_1\oplus D_2$ be a g.s. such that $D$ has no good pair.
As a consequence of Definition \ref{definition parallel} and Remark \ref{remark: existence parallel class}, we have that for any $(u,v)\in E_\oplus$, the following assertions hold:  
(i) $(u,v)\in P_{u\to v}$; (ii) $\left\vert P_{u\to v} \right\vert=\lcm(n,\ m)$; (iii) if $P_{u\to v}\cap P_{w\to z}\neq \emptyset$, then $P_{u\to v}= P_{w\to z}$; (iv) if $u\in V(C_i)=V(D_i)$ and $v\in V(C_{3-i})=V(D_{3-i})$, then $P_{u\to v}\subseteq (D_i,D_{3-i})$.

Observe that, whenever $n\neq m$ it may be more than one arc in $P_{u\to v}$ incident with $u$ or $v$.
\end{remark}

\begin{figure}
\begin{center}   
\includegraphics[scale=0.9]{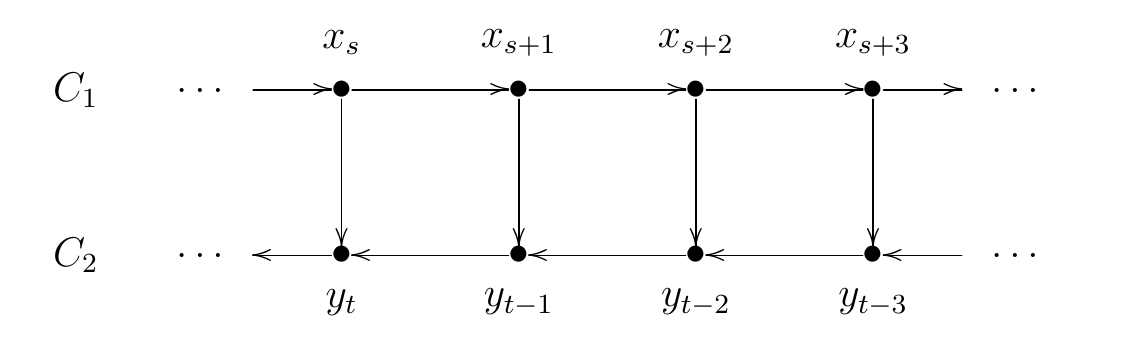}
\caption{The parallel class of $x_s\to y_t$.}
\label{fig:parallel class}
\end{center} 
\end{figure}

\begin{notation}
Let $C$ be a cycle, for each $v\in V(C)$, we will denote by $v^+$ the \emph{successor} of $v$ in $C$ (resp. $v^-$ the \emph{predecessor} of $v$ in $C$), this is, the vertex in $C$ such that $(v,v^+)\in A(C)$  (resp. $(v^-,v)\in A(C)$).
If more that one cycle contains $v$, we will write $v^{+_C}$ (resp. $v^{-_C}$). 
\end{notation}

\begin{notation}
Let $D_1$, $D_2$, \ldots, $D_k$ be a collection of pairwise vertex disjoint digraphs and take $D$ in $\oplus_{i=1}^k D_i$. For each $v\in V(D)$, we will denote by $d_\oplus^+(v)$ (resp. $d_\oplus^{-}(v)$) the number of exterior arcs in $D$ whose tail (resp. head) is $v$.
\end{notation}

The next three results will help us to prove Theorem \ref{theorem strong and no good pair then vertex pancyclic}, which asserts that, if $D$ is a strong digraph in the g.s. of two Hamiltonian digraphs, $D_1$ and $D_2$, and $D$ has no good pair of arcs between their Hamiltonian cycles, then $D$ is vertex-pancyclic.

\begin{lemma}
\label{lemma non-singular vertices, directed triangle}
Let $D_1$ and $D_2$ be two digraphs with Hamiltonian cycles, $C_1=(x_0$, $x_1$, \ldots, $x_{n-1}$, $x_0)$ and $C_2=(y_0$, $y_1$, \ldots, $y_{m-1}$, $y_0)$, respectively, and $D\in D_1\oplus D_2$. If $D$ is strong and has no good pair, then there exist $r\in [0,\ m-1]$ and $s\in[0,\ n-1]$ such that $(x_0$, $y_r$, $y_{r+1}$, $x_0)$ and $(y_0$, $x_s$, $x_{s+1}$, $y_0)$ are 3-cycles.
\end{lemma}
\begin{proof}
Since $D$ is strong, $(D_1,D_2)$ and $(D_2,D_1)$ are both non-empty.
Let $(u_0,v_0)\in (D_1,D_2)$ and $(z_0,w_0)\in (D_2,D_1)$. As there is no good pair, Remark \ref{remark: existence parallel class} implies that we can construct the parallel classes, $P_{u_0\to v_0}$ and $P_{z_0\to w_0}$. By Remark \ref{remark: existence parallel class}, for each vertex $w\in V(D)$ and each parallel class $P_{u\to v}$ there exists an arc $(u',v')\in P_{u\to v}$ incident with $w$. In particular, this holds for $x_0$,  $P_{u_0\to v_0}$ and $P_{z_0\to w_0}$. So, there are arcs $(u',v')\in P_{u_0\to v_0}$ and $(z',w')\in P_{z_0\to w_0}$, both incident with $x_0$. By Remark \ref{remark parallel class} (iv), we have that $(u',v')\in (D_1,D_2)$ and $(z',w')\in (D_2,D_1)$ as $(u_0,v_0)\in (D_1,D_2)$ and $(z_0,w_0)\in (D_2,D_1)$. Then $(u',v')=(x_0, y_a)$ and $(z',w')=(y_b,x_0)$ for some $\{a, b\}\subset [0,\ m-1]$.

In this way, the set $\{j\in[1,\ m-1]\ |\ (y_{a+j},x_0)\in A(D)\}$ is not empty (as $(y_b,x_0)\in A(D)$). Denote by $t=\min\{j\in[1,\ m-1]\ |\ (y_{a+j},x_0)\in A(D)\}$, then $(y_{a+t},x_0)\in A(D)$ and $(y_{a+t-1},x_0)\notin A(D)$, since $x_0$ and $y_{a+t-1}$ are adjacent in $D$, we have $(x_0, y_{a+t-1})\in A(D)$ and thus $(x_0$, $y_{a+t-1}$, $y_{a+t}$, $x_0)$ is a 3-cycle in $D$.




In a completely similar way, we can prove the existence of the other 3-cycle.
\end{proof}

\begin{proposition}
\label{proposition no good pair then almost vertex pancyclic 1}
Let $D_1$ and $D_2$ be two digraphs with Hamiltonian cycles, $C_1=(x_0$, $x_1$, \ldots, $x_{n-1}$, $x_0)$ and $C_2=(y_0$, $y_1$, \ldots,  $y_{m-1}$, $y_0)$, respectively, and $D\in D_1\oplus D_2$. If $D$ is strong and has no good pair, then for each vertex $v\in V(G)$ and each $t\in [3,\ 2m_1]$, there is a cycle of length $t$  passing through $v$, where $m_1=\min\{n,\ m\}$.
\end{proposition}
\begin{proof}
Suppose w.l.o.g. that $n\geq m$ and let $l=\lcm(n,\ m)$.

By Lemma \ref{lemma non-singular vertices, directed triangle}, there exists $j\in [0,\  m-1]$ such that $(x_0$, $y_{j}$, $y_{j+1}$, $x_0)$ is a 3-cycle. Assume w.l.o.g. $j=0$, this is $(x_0$, $y_0$, $y_1$, $x_0)$ is a 3-cycle in $D$.

By Remark \ref{remark: existence parallel class},  $P_{x_0\to y_0}\subseteq A(D)$ and $P_{y_{1}\to x_0}\subseteq A(D)$ 
(Figure \ref{fig plenty 3-cycles}).

\begin{figure}[!h]
\begin{center}
\includegraphics[width=\textwidth]{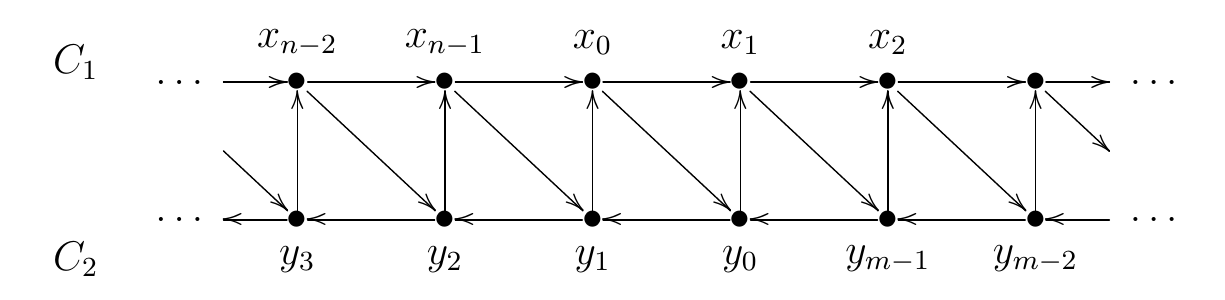}
\caption{Parallel classes of $(x_0,y_0)$ and $(y_{1},x_0)$.} 
\label{fig plenty 3-cycles}
\end{center} 
\end{figure}

\begin{claim}
\label{parallel negative}
The arc $y_{-j}\to x_{j+1}$ is in $P_{y_{1}\to x_0}$ for each $j\geq 0$.
\end{claim}
\begin{proof}[Proof of Claim \ref{parallel negative}]
By Definition \ref{definition parallel}, $P_{y_{1}\to x_0}=\{y_{1+i}\to x_{-i}\ |\ 0\leq i\leq l-1\}$.

\begin{enumerate}[{Case }1:]
\item{$j\leq l-1$.} Take $i=l-1-j$. As $1\leq 1+j \leq l$, we have that $0\leq i = l-1-j\leq l-1$. Hence, $(y_{1+i},\ x_{-i})=(y_{1+(l-1-j)},\ x_{-(l-1-j)})=(y_{-j},\ x_{1+j})$ is an arc in $P_{y_{1}\to x_0}$. 

\item{$j\geq l$.} By Euclidean algorithm, there exist $q'\geq 0$ and $r'\in[0,\ l-1]$ such that $j=q'l+r'$. Then $(y_{-j},\ x_{j+1})=(y_{-(q'l+r')},\ x_{q'l+r'+1})=(y_{-r'},\ x_{r'+1})$, as $l=\lcm(n,\ m)$. Since $0\leq r' \leq l-1$ we have that $(y_{-r'},\ x_{r'+1})\in P_{y_{1}\to x_0}$, by Case 1.
\end{enumerate} 
\end{proof}

For each  $t\in[0,\ n-1]$ and each $h\in[0,\ m-2]$, consider the cycle $\alpha_{h}^t$ $=$ $(x_t,\ y_{-t})$ $\cup$ $(y_{-t},\ C_2,\ y_{-t+h+1})$ $\cup$ $(y_{-t+h+1},\  x_{t-h})$ $\cup$ $(x_{t-h},\ C_1,\ x_t)$.

For each $h\in [0,\ m-2]$, we have the cycle $\alpha_{h}^t$ with length $l(\alpha_{h}^t)=1+(h+1)+1+h=2h+3$. Therefore, for each $t\in[0,\ n-1]$, there are cycles of each odd length from 3 to $2(m-2)+3=2m-1$ passing through $x_t$ and $y_{-t}$.
Recall that $n\geq m$, so we have proved that for each $v\in V(D)$ there  are cycles of each odd length in $[3,\ 2m-1]$ passing through $v$.

\begin{enumerate}[{Case }1:]
\item{$(x_0, y_{2})\in A(D)$.}
Then $P_{x_0 \to y_{2}}\subseteq A(D)$, where $P_{x_0 \to y_{2}}=\{x_i\to y_{2-i}\ |\ 0\leq i\leq l-1\}$. 

For each  $t\in[0,\ n-1]$ and each $h\in[0,\ m-2]$, consider the cycle $\beta_{h}^t$ $=$ $(x_t,\ y_{2-t})$ $\cup$ $(y_{2-t},\ C_2,\ y_{2-t+h})$ $\cup$ $(y_{2-t+h},\  x_{-1+t-h})$ $\cup$ $(x_{-1+t-h},\ C_1,\ x_{t-1})$ $\cup$ $(x_{t-1}, y_{-t+1}, x_t)$ (Figure \ref{fig beta cycle}).

\begin{figure}[!h]
\begin{center}
\includegraphics[width=\textwidth]{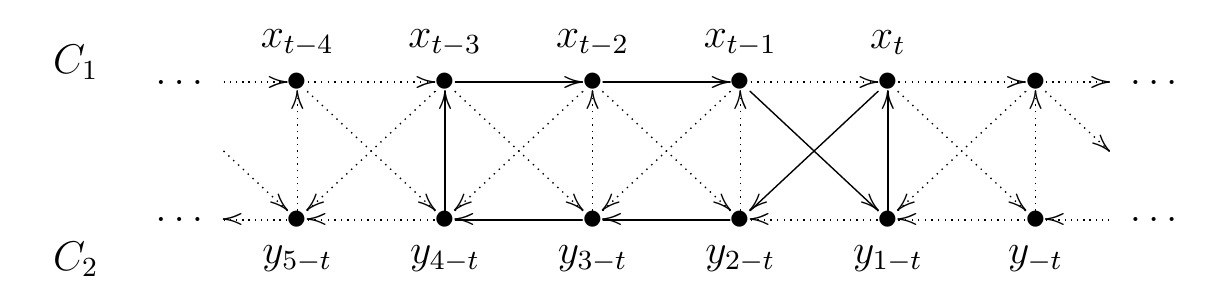}
\caption{A $\beta_2^t$ cycle in $D$.} 
\label{fig beta cycle}
\end{center} 
\end{figure}

For each $h\in [0,\ m-2]$, we have the cycle $\beta_{h}^t$ with length $l(\beta_{h}^t)=1+h+1+h+2=2h+4$. Therefore, for each $t\in[0,\ n-1]$, there are cycles of each even length from 4 to $2(m-2)+4=2m$ passing through $x_t$ and $y_{2-t}$.
As $n\geq m$, we have proved that for each $v\in V(D)$ there are cycles of each even length in $[4,\ 2m]$ passing through $v$.

\item{$(y_{2},x_0)\in A(D)$.}
Then $P_{y_{2}\to x_0}\subseteq A(D)$, where $P_{y_{2}\to x_0}=\{y_{2+i}\to x_{-i}\ |\ 0\leq i\leq l-1\}$. 

For each  $t\in[0,\ n-1]$ and each $h\in[0,\ m-3]$, consider the cycle $\gamma_{h}^t$ $=$ $(x_t,\ y_{-t}) \cup (y_{-t},\ C_2,\ y_{-t+h+2}) \cup (y_{-t+h+2},\  x_{t-h}) \cup (x_{t-h},\ C_1,\ x_t)$.

For each $h\in [0,m-3]$, we have the cycle $\gamma_{h}^t$ with length $l(\gamma_{h}^t)=1+(h+2)+1+h=2h+4$. Therefore, for each $t\in[0,\ n-1]$, there are cycles of each even length from 4 to $2(m-3)+4=2m-2$ passing through $x_t$ and $y_{-t}$.
So, we proved that for each $v\in V(D)$ there are cycles of each even length in $[4,\ 2m-2]$ passing through $v$.
\end{enumerate}

Finally, we show that for each $v\in V(D)$ there is a cycle of length $2m$ passing through $v$.

For each  $t\in[0,\ n-1]$, take $\varepsilon^t$ $=$ $(y_{-(t-1)}$, $x_t$, $y_{-t}$, $x_{t+1}$, $y_{-(t+1)}$, $x_{t+2}$, \ldots, $y_{-(t+m-2)}$, $x_{t+m-1}$, $y_{-(t+m-1)})$, which is a cycle as $y_{-(t+m-1)}=y_{-(t-1)}$, moreover,  $l(\varepsilon^t)=2m$. 

Then, for each $t\in[0, n-1]$, there is a cycle of length $2m$ passing through $x_t$ and $y_{-(t-1)}$. We conclude, for each $v\in V(D)$, that there are cycles of every length in $[3,\ 2m]$ passing through $v$.

\end{proof}

\begin{proposition}
\label{proposition no good pair then almost vertex pancyclic 2}
Let $D_1$ and $D_2$ be two digraphs with Hamiltonian cycles, $C_1=(x_0$, $x_1$, \ldots, $x_{n-1}$, $x_0)$ and $C_2=(y_0$, $y_1$, \ldots,  $y_{m-1}$, $y_0)$, respectively, and $D\in D_1\oplus D_2$. If $D$ is strong and contains no good pair, then for each vertex $v\in V(G)$ there is a cycle, of each length in $[m_1+2,\ n+m]$ passing through $v$, where $m_1=\min\{n,\ m\}$.
\end{proposition}
\begin{proof}

If $n=m$, then Proposition \ref{proposition no good pair then almost vertex pancyclic 1} asserts that for each $v\in V(D)$ there is a cycle of each length in $[3,\ 2m]=[3,\ n+m]$ passing through $v$. 
Hence, we assume $n\neq m$. 
Suppose w.l.o.g. that $n > m$ and let $l$ be the least common multiple of $n$ and $m$, $l=\lcm(n,\ m)$. 

By Lemma \ref{lemma non-singular vertices, directed triangle}, there exists $j\in [0,\ m-1]$ such that $(x_0$, $y_{j}$, $y_{j+1}$, $x_0)$ is a 3-cycle in $D$. Assume w.l.o.g. that $j=0$, this is $(x_0$, $y_0$, $y_1$, $x_0)$ is a 3-cycle in $D$.

By Remark \ref{remark: existence parallel class},  $P_{x_0\to y_0}$ and $P_{y_{1}\to x_0}\subseteq A(D)$, where $P_{x_0\to y_0}=\{x_i\to y_{-i}\ |\ 0\leq i\leq l-1\}$ and $P_{y_{1}\to x_0}=\{y_{1+i}\to x_{-i}\ |\ 0\leq i\leq l-1\}$ 
(Figure \ref{fig plenty 3-cycles}) and, by Claim \ref{parallel negative} in the proof of Proposition \ref{proposition no good pair then almost vertex pancyclic 1}, we have that, for each $t\in [0,\ n-1]$, $y_{-t}\to x_{t+1}$ is an arc in $P_{y_1\to x_0}$.

For each $t\in [0,\ n-1]$ and each $j\in [1,\ m-1]$, define a path $P_{j}^t$ as follows: $P_{j}^t=(y_{-t}, x_{t+1}$, $y_{-(t+1)}$, $x_{t+2}$, \ldots, $y_{-(t+j)}$, $x_{t+j+1})\cup (x_{t+j+1},\ C_1,\ x_{t+m})$. It has length $l(P_{j}^t)= (2j+1)+(m-j-1)=m+j$.

Since $n > m$ we have, by Euclidean algorithm, that there exist $q\geq 1$ and $r\in [0,\ m-1]$ such that $n=qm+r$.

For each $t\in [0,\ n-1]$, consider the arcs of the form $x_{t+im}\to y_{-(t+im)}$ in $P_{x_0\to y_0}$, where $1\leq i \leq q$. As $C_2$ has length $m$ and we take subscripts modulo $m$, for each $i\in[1,\ q]$, $y_{-(t+im)}=y_{-t}$, moreover, the fact that $n=qm+r$ with $q\geq 1$ and $r\in [0,\  m-1]$, implies that the vertices $x_{t+m}$, $x_{t+2m}$, \ldots, $x_{t+qm}$ are $q$ different vertices in $C_1$, and thus the arcs  $x_{t+m}\to y_{-t}$, \ldots, $x_{t+qm}\to y_{-t}$ are $q$ different arcs in $P_{x_0\to y_0}$.



Now, for each $t\in [0,\ n-1]$, each $i\in[1,\ q]$ and each $j\in [1,\ m-1]$, define a cycle $\beta_{i,j}^t$ as follows:
$\beta_{i,j}^t$ $=$ $P_{j}^t \cup (x_{t+m},\ C_1,\ x_{t+im})$ $\cup$ $(x_{t+im},\ y_{-t})$. It has length $l(\beta_{i,j}^t)=(j+m)+(im-m)+1=im+j+1$ and passes through $y_{-t}$ and $x_{t+1}$. Notice that, as $n>m$ and $0\leq t \leq n-1$, we have that for each $v\in V(D)$ and each $h\in [m+2,\ (q+1)m]=[m+2,\ n-r+m]$ there is a cycle of length $h$ passing through $v$.

Finally, for each $t\in [0,\ n-1]$ and each $j\in [1,\ m-1]$, define a cycle $\gamma_{j}^t$ as follows:
$\gamma_j^t$ $=$ $P_{j}^t \cup (x_{t+m},\ C_1,\ x_{t})$ $\cup$ $(x_{t},\ y_{-t})$. It has length $l(\gamma_j^t)=(m+j)+(n-m)+1=n+j+1$ and passes through $y_{-t}$ and $x_{t+1}$. And, since $n>m$ and $0\leq t \leq n-1$, we have that for each $v\in V(D)$ and each $h\in [n+2,\ n+m]$ there is a cycle of length $h$ passing through $v$, concluding the proof.

\end{proof}

As a direct consequence of Propositions \ref{proposition no good pair then almost vertex pancyclic 1} and \ref{proposition no good pair then almost vertex pancyclic 2} we have the following:

\begin{theorem}
\label{theorem strong and no good pair then vertex pancyclic}
Let $D_1$ and $D_2$ be two digraphs with Hamiltonian cycles, $C_1=(x_0$, $x_1$, \ldots, $x_{n-1}$, $x_0)$ and $C_2=(y_0$, $y_1$, \ldots,  $y_{m-1}$, $y_0)$, respectively, and $D\in D_1\oplus D_2$. If $D$ is strong and contains no good pair, then $D$ is vertex-pancyclic.
\end{theorem}

In Section \ref{main results}, we will extend this result to a strong digraph $D$ in the g.s. of $k$ vertex disjoint Hamiltonian digraphs. First we will prove Lemma \ref{lemma 3-cycle vertex-pancyclic} and then we will define an anti-directed cycle, and observe that a good pair is an anti-directed 4-cycle with some properties.

\begin{lemma}
\label{lemma 3-cycle vertex-pancyclic}
Let $D_1$, $D_2$, $D_3$ be three vertex disjoint digraphs with Hamiltonian cycles, $C_1$, $C_2$, $C_3$, respectively, and $D\in \oplus_{i=1}^3 D_i$. If there is a sequence $\{i_j\}_{j=1}^3$ such that $D_{i_1} \mapsto D_{i_2}$, $D_{i_2} \mapsto D_{i_3}$ and $D_{i_3} \mapsto D_{i_1}$ in $D$, then $D$ is vertex-pancyclic.
\end{lemma}
\begin{proof}
Suppose w.l.o.g. that $D_{1} \mapsto D_{2}$, $D_{2} \mapsto D_{3}$ and $D_{3} \mapsto D_{1}$ in $D$.

Let $C_1=(x_0$, $x_1$, \ldots, $x_{n-1}$, $x_0)$, $C_2=(y_0$, $y_1$, \ldots,  $y_{m-1}$, $y_0)$ and $C_3=(w_0$, $w_1$, \ldots, $w_{t-1}$, $w_0)$ be the Hamiltonian cycles of $D_1$, $D_2$ and $D_3$, respectively.

As $D_{1} \mapsto D_{2}$, $D_{2} \mapsto D_{3}$ and $D_{3} \mapsto D_{1}$ in $D$ we have, for each $x\in V(C_1)$, each $y\in V(C_2)$ and each $w\in V(C_3)$ that $x\to y$, $y\to w$ and $w\to x$.

For each $i\in [0,\ n-1]$, each $j\in [0,\ m-1]$, each $h\in[0,\ t-1]$ and each $g\in [0,t-1]$, consider the cycle $\alpha(i,j,h,g)=(x_i$, $y_j$, $w_h) \cup (w_h$, $C_3$, $w_{h+g}) \cup (w_{h+g}$, $x_i)$, it has length $l(\alpha(i,j,h,g))=2+g+1=g+3$ and passes through $x_i$, $y_j$ and $w_h$. Then, for each $v\in V(D)$ and each length $l$ in $[3,\ t+2]$ there is a cycle of length $l$ passing through $v$.
    
For each $i\in [0,\ n-1]$, each $j\in [0,\ m-1]$ and each $g'\in[0,\ m-1]$, consider the cycle $\beta(i,j,g')=(x_i$, $y_j) \cup (y_j$, $C_2$, $y_{j+g'}) \cup (y_{j+g'}$, $w_0) \cup (w_0$, $C_3$, $w_{t-1}) \cup (w_{t-1}$, $x_i)$; it has length $l(\beta(i,j,g'))=1+g'+1+(t-1)+1=g'+t+2$ and passes through $x_i$, $y_j$ and each $w\in V(C_3)$. Then, for each $v\in V(D)$ and each length $l$ in $[t+2,\ m+t+1]$ there is a cycle of length $l$ passing through $v$.

For each $i\in [0,\ n-1]$ and each $g''\in [0,\ m-1]$, consider the cycle $\gamma(i,g'')=(x_i$, $C_1$, $x_{i+g''}) \cup (x_{i+g''}$, $y_0) \cup (y_0$, $C_2$, $y_{m-1}) \cup (y_{m-1}$, $w_0) \cup (w_0$, $C_3$, $w_{t-1}) \cup (w_{t-1}$, $x_i)$; it has length $l(\gamma(i,g''))=g''+1+(m-1)+1+(t-1)+1=g''+m+t+1$ and passes through $x_i$, each $y\in V(C_2)$ and each $w\in V(C_3)$. Then, for each $v\in V(D)$ and each length $l$ in $[m+t+1,\ n+m+t]$ there is a cycle of length $l$ passing through $v$.
    
Therefore, $D$ is vertex-pancyclic.
\end{proof}

Let $D$ be a digraph. A succession of vertices  $\mathcal{C}=v_0v_1\cdots v_{t-1}v_0$ is an \emph{anti-directed $t$-cycle} whenever $v_i\neq v_j$ for each $i\neq j$, $t$ is even and, for each $i\equiv 0 \pmod{2}$, $\{(v_i, v_{i+1})$, $(v_i, v_{i-1})\}\subset A(D)$ or $\{(v_i, v_{i+1})$, $(v_i, v_{i-1})\}\subset A(D)$. \\

We may assume that every anti-directed cycle starts with a forward arc, else we might relabel the subscripts. \\

From the definitions of good pair of arcs and anti-directed cycle we obtain the following remark.

\begin{remark}
\label{remark:good pair/anti-directed 4-cycle}
Let $D_1$ and $D_2$ be two vertex disjoint digraphs, $\alpha_1=(u_0$, $u_1$, \ldots, $u_{p-1}$, $u_0)$ and $\alpha_2=(v_0$, $v_1$, \ldots, $v_{q-1}$, $v_0)$ be two cycles in $D_1$ and $D_2$, respectively, and let $D \in D_1\oplus D_2$. If $u_s\to v_r$ and $v_{r-1}\to u_{s+1}$ is a good pair of arcs. Then $\mathcal{C}=u_s v_r v_{r-1} u_{s+1} u_s$ is an anti-directed 4-cycle whose arcs belong alternatively to $E_\oplus$ and $A(\alpha_1)\cup A(\alpha_2)$, namely $\{(u_s, v_r)$, $(v_{r-1}, u_{s+1})\}\subset E_\oplus$,  $(u_s, u_{s+1})\in A(\alpha_1)$ and $(v_{r-1}, v_r)\in A(\alpha_2)$.
\end{remark}

\begin{definition}
\label{def:good anti-directed cycle}
Let $D_1$, $D_2$, \ldots, $D_k$ be a collection of pairwise vertex disjoint digraphs, and $D\in \oplus_{i=1}^k D_i$. 
An anti-directed 4-cycle $\mathcal{C}= v_0 v_1 v_2 v_{3}v_0$ in $D$ will be called a \emph{good cycle} whenever at least one of the following conditions holds $\{(v_0,v_1)$, $(v_2,v_3)\}\subset E_\oplus$ or $\{(v_2,v_1)$, $(v_0,v_3)\}\subset E_\oplus$. 
\end{definition}

\section{Main results}
\label{main results}
In this section we generalize Theorem  \ref{theorem strong and no good pair then vertex pancyclic} for a strong digraph $D$ in the g.s. of $k$ vertex disjoint Hamiltonian digraphs.

\begin{theorem}
\label{theorem strong and no antidirected 4-cycle then vertex pancyclic}
Let $D_1$, $D_2$, \ldots, $D_k$ be a collection of $k\geq 2$ vertex disjoint digraphs with Hamiltonian cycles, $C_1$, $C_2$, \ldots, $C_k$, respectively, and $D\in \oplus_{i=1}^k D_i$. If $D$ is strong and contains no good cycle, then $D$ is vertex-pancyclic.
\end{theorem}
\begin{proof}
We proceed by induction on $k$.

When $k=2$, $D$ has no good pair with respect to $C_1$ and $C_2$, as $D$ contains no good cycle. Otherwise, by Remark \ref{remark:good pair/anti-directed 4-cycle}, $D$ would contain a good cycle. Hence, $D$ satisfies the hypothesis of Theorem \ref{theorem strong and no good pair then vertex pancyclic} and $D$ is vertex-pancyclic.

Now, assume that the assertion holds for each  $k'\in [2,\ k-1]$. We prove that the assertion holds for $k$.

Let $D_1$, $D_2$, \ldots, $D_k$ be a collection of $k$ vertex disjoint digaphs with Hamiltonian cycles, $C_1$, $C_2$, \ldots, $C_k$, respectively, and take $D$ a strong digraph with no good cycle in $\oplus_{i=1}^{k} D_i$.

\begin{enumerate}[{Case }1:]
\item There exist two different indices $\{j, j'\}\subset [1,\ k]$ such that $(D_j, D_{j'})\neq \emptyset$ and $(D_{j'}, D_{j})\neq \emptyset$ (therefore, $D_{0}=D\langle V(D_j)\cup V(D_{j'})\rangle$ is strong).

As $D_{0}$ is a subdigraph of $D$, it contains no good cycle. Hence, $D_0$ satisfies the induction basis and thus it is vertex-pancyclic. In particular, $D_0$ contains a Hamiltonian cycle $C_0$.

Let $J=[0,\ k]\setminus \{j$, $j'\}$ and observe that $\{D_i\}_{i\in J}$ is a collection of $k-1$ pairwise vertex disjoint Hamiltonian digraphs. Notice that $D\in \oplus_{i\in J} D_i$ (as $D$ satisfies Definition \ref{definition g.s.}).

Therefore, by the inductive hypothesis, $D$ is vertex-pancyclic.

\item For each pair of different indices $\{i,j\}\subset [1,\ k]$, either $D_i \mapsto D_{j}$ or $D_{j} \mapsto D_{i}$ in $D$. 

Define the digraph $H$ with vertex set $V(H)=\{D_1$, $D_2$, \ldots, $D_k\}$ and $(D_i, D_{j})\in A(H)$ iff $D_i \mapsto D_{j}$ in $D$.  Clearly $H$ is a strong tournament and thus, for each $t\in[3,\ k]$, $H$ contains a cycle $\alpha_t$ of length $t$ \cite{Moon1966297}. 

If $k=3$, then $\alpha_t$ is a Hamiltonian cycle in $H$. Hence, $D$ satisfies the hypothesis of Lemma \ref{lemma 3-cycle vertex-pancyclic} and thus $D$ is vertex-pancyclic. 

If $k\geq 4$, then $H$ contains a non-Hamiltonian cycle. Let $\alpha_t=(D_{i_{0}}$, $D_{i_{1}}$, \ldots, $D_{i_{t-1}}$, $D_{i_0})$ be a cycle in $H$ of length $t$, for some $t\in[3,\ k-1]$.

Recall that $D_i$ contains a Hamiltonian cycle $C_i$, for each $i\in [1,\ k]$. Take, for each $j\in [0,\ t-1]$, a Hamiltonian path in $D_{i_j}$, namely $T_{i_j}$, and call $u_{j}$ and $v_j$ the initial and final vertices of $T_{i_j}$, respectively. Since $T_{i_j}\mapsto T_{i_{j+1}}$ for each $j\in[0,\ t-1]$, we have that $(v_j, u_{j+1})\in A(D)$ for each $j\in[0,\ t-1]$,  where the subscripts are taken modulo $t$.

Therefore, $\gamma= T_{i_{0}}\cup (v_0,u_1) \cup T_{i_{1}} \cup (v_1,u_2)\cup \cdots  \cup T_{i_{t-1}} \cup (v_{t-1}, u_0)$ is a cycle  with $V(\gamma)=\bigcup_{j=0}^{t-1} V(D_{i_j})$.

Let $D_0=D\langle \bigcup_{j=0}^{t-1} V(D_{i_j}) \rangle$ and $J=[0,\ k]\setminus \{i_0$, $i_1$, \ldots, $i_{t-1}\}$, then $D\in \oplus_{i\in J} D_i$ (recall that $\oplus$ is associative and $D$ satisfies Definition \ref{definition g.s.}).

Observe that $\{D_i\}_{i\in J}$ is a collection of $\vert J\vert= k+1-t$ vertex disjoint Hamiltonian digraphs, where $2\leq k+1-t < k$, and $D\in \oplus_{i\in J} D_i$ has no good cycle, since the set of exterior arcs of $D$ as g.s. of $\{D_i\}_{i\in J}$ is contained in the set of exterior arcs of $D$ as a g.s. of $D_1$, $D_2$, \ldots, $D_k$. Therefore, by the inductive hypothesis, $D$ is vertex-pancyclic.

\end{enumerate}

\end{proof}

In this paper we considered generalized sums of Hamiltonian digraphs containing no good pair (or its generalization, good cycle) and obtained conditions that are easy to check to determine if a g.s. is vertex-pancyclic. The study of the case when good pairs appear will be treated in a forthcoming paper.


\bibliography{Pancyclism-in-the-g.s.}


\end{document}